%% file: main.tex
\documentclass{amsart}
\usepackage{my_preamble}

\title[Degenerations of moduli stacks of sheaves]{Degenerations of shifted symplectic moduli stacks of sheaves}
\author{Matthew Huynh}
\date{\today}

\begin{document}

\begin{abstract}
    Let $S$ be a smooth, affine noetherian $\CC$-scheme, and let $X/S$ be a smooth, quasi-projective family of Calabi--Yau $d$-folds.
    We prove that the derived moduli stack of flat families of properly supported coherent sheaves on $X/S$ has a relative $(2-d)$-shifted symplectic structure, by generalizing a construction of \cite{brav_dyckerhoff_2}.
    The main points are to show that the category $\rm{IndCoh}(X)$ is a smooth, compact $\rm{QCoh}(S)$-module and to identify our moduli stack of interest as a substack of the moduli of pseudo-perfect objects in $\rm{IndCoh}(X)$.
    Finally, we discuss two examples of families of smooth, quasi-projective Calabi--Yau threefolds that are potentially useful for applications in enumerative geometry.
\end{abstract}

\maketitle

\tableofcontents


\setcounter{section}{-1}
\input{results/intro}
\input{results/linear_categories}
\input{results/moduli_of_objects}
\input{results/forms}
\input{results/sheaves.tex}

\section*{Statement on use of AI}
The author did not use generative AI for assistance in the research nor writing of this note.

\section*{Acknowledgments}
The author thanks Ceyhun Elmac{\i}o\u{g}lu, Sonja Farr, Joj Helfer, Hyeonjun Park, Wyatt Reeves, Jemin You, and Frank Zhang for helpful discussions during various stages of writing this note.
The author is supported by KIAS Individual Grant MG111001. 

\bibliography{references2}
\footnotesize{
    \vspace{\baselineskip}
    \noindent\textsc{School of Mathematics, Korea Institute for Advanced Study, 85 Hoegiro, Dongdaemun-gu, Seoul 02455, Republic of Korea} \par\nopagebreak
    \vspace{\baselineskip}
    \noindent\textit{Email address}:
    \texttt{mhuynh@kias.re.kr}
}

\end{document}

%% file: results/intro.tex
\section{Introduction}
Let $S$ be a smooth, affine, noetherian $\CC$-scheme and let $X/S$ be a smooth, quasi-projective family of Calabi--Yau $d$-folds.
The purpose of this note is to prove that there exists a relative $(2-d)$-shifted symplectic structure on the derived moduli stack $\M_{X/S}$ of flat families of properly supported coherent sheaves on $X/S$.

When $X/S$ is \emph{projective}, the seminal work of Pantev--To\"en--Vacqui\'e--Vezzosi \cite{ptvv} explained how to construct a shifted symplectic form via the following procedure. 
First, recognize the moduli stack of perfect complexes as a mapping stack $\maps(X,\rm{perf}_S)$, then pull back a natural $2$-shifted symplectic form from $\rm{perf}_S$ to $X \times_S \maps(X,\rm{perf}_S)$, and finally integrate the pulled-back form in the $X$-direction.
Later on, Brav--Dyckerhoff \cites{brav_dyckerhoff_1, brav_dyckerhoff_2} explained how to construct a shifted symplectic form on $\M_{X/S}$ when $S = \Spec(\CC)$ and $X$ is quasi-projective.
Really, this note shows that some results of Brav--Dyckerhoff work for more general base schemes $S$.

There are two kinds of families of Calabi--Yau \emph{threefolds} that are particularly interesting to us and for which we intend to apply our main result in later works.
We obtain the first kind of family by fixing a smooth subvariety $Z$ inside a smooth projective CY3 $Y$, and then taking $X/\AA^1$ to be the (classical) deformation to the normal cone.
We obtain the second kind of family by fixing a smooth complex projective curve $C$ and an $\AA^1$-family of rank two vector bundles $V_t$ on $C$ satisfying $\det(V_t) \cong \Omega^1_C$, and then taking $X/\AA^1$ to be the total spaces of these bundles.

We consider our main result applied to the first kind of example to be a derived and higher-dimensional generalization of a construction of Donagi--Ein--Lazarsfeld \cite{donagi_ein_lazarsfeld},
who degenerated the moduli space of sheaves on a K3 surface to the moduli space of Higgs bundles on a curve appearing in a linear system on the K3.
Their construction was later generalized by replacing the K3 surface with other symplectic or Poisson surfaces, see e.g. \cites{franco_degenerations, zhao2025moduli}.
This symplectic degeneration of moduli spaces was used by de Cataldo--Maulik--Shen \cite{decataldo_maulik_shen} to prove the $P=W$ Conjecture for genus $2$ curves.
Roughly speaking, these authors deformed the moduli space of Higgs bundles on such a curve to a moduli space of sheaves on its Jacobian, which allowed them to access tools from hyperk\"ahler geometry and compare the perversity of tautological cohomology classes on the two moduli spaces. 

We consider families of local curves interesting because of their potential application to proving a certain correspondence between Gopakumar--Vafa invariants and Gromov--Witten invariants. 
Indeed, Pardon's proof \cite{pardon_mnop} of the MNOP Conjecture relating GW, DT, and PT invariants highlights the importance of local curves when comparing enumerative invariants.

The mathematical definition of Gopakumar--Vafa invariants proposed by Maulik--Toda \cite{maulik_toda} goes roughly as follows.
Fix a smooth quasi-projective Calabi--Yau threefold $X$ and a curve class $\beta$, and consider the derived moduli stack of sheaves on $X$ whose support has homology class $\beta$.
The moduli stack of sheaves has a $(-1)$-shifted symplectic structure, so it can be covered by charts that are critical loci of locally-defined functions \cite{bbbj_darboux_stack}.
One may glue together the sheaves of vanishing cycles to obtain the so-called DT sheaf $\varphi$ on the moduli stack of sheaves on $X$ \cite{bbdj_joyce_sheaf}.
(For the purposes of \cite{maulik_toda}, they require an additional piece of structure called \emph{Calabi--Yau orientation data}, but we do not dwell on this here).
Finally, one uses the Hilbert--Chow morphism that sends a sheaf to the class of its Fitting support to endow the cohomology of $\varphi$ with the perverse-Leray filtration.
Taking the Euler characteristic of the graded pieces of this filtration determines the Gopakumar--Vafa invariants of $X$ and $\beta$.

In order to compare GV invariants of different local curves, we would like to generalize the structures and data in the paragraph above by working over a base $S$.
This note records the fact that the moduli stack of sheaves on a family of local curves $X/S$ admits a relative $(-1)$-shifted symplectic structure, which is a first step towards our goal.
We plan to compare critical charts for families of local curves in a later work.

\subsection{Brief outline}
The first three sections are preliminary.
In \S \ref{sect:linear_categories}, we review relevant notions about $\V$-module categories like finiteness conditions on them, their negative cyclic and Hochschild homology, and (left) Calabi--Yau structures.
In \S \ref{sect:moduli_of_objects}, we recall the definition of the moduli stack of ``perfect'' objects in an $A$-linear $\infty$-category $\C$ and the Representability Theorem of \cite{toen_vacquie_moduli_of_objects}.
In \S \ref{sect:forms}, we explain, in the language of stable $\infty$-categories, the main construction of \cite{brav_dyckerhoff_2}, which produces a shifted symplectic form on the moduli of objects in $\C$ from the data of a left Calabi--Yau structure on $\C$.

In \S \ref{sect:sheaves}, we apply this discussion to $\C = \rm{IndCoh}(X)$ considered as a $\rm{QCoh}(S)$-module.
The main points are to verify that $\C$ is a finite type $\rm{QCoh}(S)$-module and to identify the moduli stack of flat families of properly supported coherent sheaves as a substack of the moduli of objects in $\C$.

\subsection{Notation and conventions}
We work over the field of complex numbers $\CC$ throughout this article.

By default, all notions are $\infty$-categorical and derived.
\begin{itemize}
    \item We write $\rm{Ani}$ for the $\infty$-category of anima.
    \item We write $\rm{Pr}^L$ for the $\infty$-category of presentable $\infty$-categories and colimit-preserving functors between them. 
    We write $\rm{Pr}^L_{\omega}$ for the $\infty$-category of compactly generated presentable $\infty$-categories and colimit-preserving functors that also preserve compact objects.
    \item If $\V \in \rm{CAlg}(\rm{Pr}^L)$, then we define
    \[  
        \rm{Pr}^L_{\V} := \rm{Mod}_{\V}(\rm{Pr}^L_{\omega}),\quad \rm{Pr}^L_{\omega,\V} := \rm{Mod}_{\V}(\rm{Pr}^L_{\omega}).
    \]
    \item We interchangeably use $\Map_{\C}(\mhyphen,\mhyphen)$ and $\C(\mhyphen,\mhyphen)$ for mapping spaces and use $\maps(\mhyphen,\mhyphen)$ for mapping prestacks. 
    If $\C$ is a $\V$-module, then we write 
    \[
        \Hom_{\C}^{\V}(\mhyphen,\mhyphen): \C^{\op}\times\C \to \V,
    \]
    for the enriched hom.
    We use $\hom_{\C}(\mhyphen,\mhyphen)$ for internal homs.
    \item If $A$ is a connective $\EE_{\infty}$-$\CC$-algebra and $\V = \rm{Mod}_A(\rm{Sp})$, then we write $\rm{Pr}^L_{A}$ (resp. $\rm{Pr}^L_{\omega,A}$) instead of $\rm{Pr}^L_{\V}$ (resp. $\rm{Pr}^L_{\omega,\V}$).
    We refer to $\rm{Mod}_A$-module categories as $A$-linear categories and $\rm{Mod}_A$-linear functors as $A$-linear functors.
    \item We write $\rm{SchAff}$ for the $\infty$-category of affine schemes over $\Spec(\CC)$, i.e. the opposite category of the $\infty$-category of connective $\EE_{\infty}$-$\CC$-algebras.
    We add the decoration ``$<\infty$'' to denote eventually coconnective affine schemes, i.e. ones for which $\pi_iA = 0$ for all $i >> 0$, and we add the decoration ``$\rm{ft}$'' to denote finite-type affine schemes over $\Spec(\CC)$.
\end{itemize}

%% file: results/linear_categories.tex
\section{Module categories}\label{sect:linear_categories}
Let $\V$ be a presentable, symmetric monoidal $\infty$-category, i.e. an object of $\rm{CAlg}(\rm{Pr}^L)$.
First, we recall various finiteness conditions on $\V$-modules.
Then, we review the definition of negative cyclic and Hochschild homology of $\V$-modules.
Lastly, we recall the notion of left $d$-Calabi--Yau structures on $\V$-modules.
Oftentimes throughout this section, we specialize to the case when $\V = \rm{Mod}_A$, where $A$ is a connective $\EE_{\infty}$-$\CC$-algebra.

\subsection{Finiteness conditions}
The following facts may be useful to keep in mind.
Any $\V$-module $\C \in \rm{Pr}^L_{\V}$ is enriched over $\V$ by \cite{heine_enriched}*{Theorem 1.2}.
Furthermore, if $\V = \rm{Mod}_A$, then any $\V$-module $\C$ is stable \cite{lurie_SAG}*{\S D.1.5}.

We say that $\C$ is a dualizable $\V$-module if it is dualizable as an object of $\rm{Pr}^L_{\V}$ in the sense of \cite{lurie_HA}*{\S 4.6.1}.
Now suppose that $\C$ is a dualizable $A$-linear category.
Then the following compositions are equivalences of $A$-linear categories \cite{christ_shobers}*{Lemma 2.12}:
\begin{multline}
    \Theta_{\C}:\\
    \xymatrixcolsep{1.8cm}\xymatrix{
        \hom(\C^{\vee}\otimes \C,\rm{Mod}_A) \ar[r]^-{\id\otimes (\mhyphen)} & \hom(\C\otimes \C^{\vee}\otimes \C,\C) \ar[r]^-{(\mhyphen)\circ(\rm{coev}\otimes\id)} & \hom(\C,\C),
    }
\end{multline}
\begin{multline}
    \Xi_{\C}: \\
    \xymatrixcolsep{1.8cm}\xymatrix{
        \hom(\rm{Mod}_A,\C\otimes\C^{\vee})\ar[r]^-{(\mhyphen)\otimes\id} & \hom(\C,\C\otimes\C^{\vee}\otimes\C)\ar[r]^-{(\id\otimes\rm{ev})\circ(\mhyphen)} & \hom(\C,\C).
    }
\end{multline}
Their inverses are given by
\begin{multline}
    \Theta^{-1}_{\C}: \\
    \xymatrixcolsep{1.8cm}\xymatrix{
        \hom(\C,\C)\ar[r]^-{\id\otimes(\mhyphen)} &\hom(\C^{\vee}\otimes\C,\C^{\vee}\otimes\C)\ar[r]^-{\rm{ev}\circ(\mhyphen)} & \hom(\C^{\vee}\otimes\C,\rm{Mod}_A),
    }
\end{multline}
respectively
\begin{multline}
    \Xi^{-1}_{\C}: \\
    \xymatrixcolsep{1.8cm}\xymatrix{
        \hom(\C,\C)\ar[r]^-{(\mhyphen)\otimes\id} & \hom(\C\otimes\C^{\vee},\C\otimes\C^{\vee})\ar[r]^-{(\mhyphen)\circ\rm{coev}}&\hom(\rm{Mod}_A,\C\otimes\C^{\vee}).
    }
\end{multline}

Write $\tau: \C\otimes\C^{\vee}\isomto \C^{\vee}\otimes \C$ for the $A$-linear involution that swaps the two factors.
\begin{defn}[=\cite{christ_shobers}*{Definition 2.14, Proposition 2.15}]
\label{defn:left_right_duals_of_endofunctors}
    Let $F \in \hom(\C,\C)$ be an $A$-linear endofunctor of a dualizable $A$-linear category $\C$.
    \begin{enumerate}
        \item Say that $F$ is \emph{left dualizable} if $\Xi_{\C}^{-1}(F)$ admits an $A$-linear right adjoint.
        Then define the \emph{left dual of $F$} by
        \[
            F^! := \Theta_{\C}(\tau\circ\Xi_{\C}^{-1}(F)^R) \in \hom(\C,\C).
        \]
        \item Say that $F$ is \emph{right dualizable} if $\Theta_{\C}^{-1}(F)$ admits an $A$-linear right adjoint.
        Then define the \emph{right dual of $F$} by
        \[
            F^{\ast} := \Xi_{\C}(\tau\circ\Theta_{\C}^{-1}(F)^R) \in \hom(\C,\C).
        \]
    \end{enumerate}
\end{defn}

Finally, we may define finiteness conditions for dualizable $A$-linear categories.
\begin{defn}[=\cite{christ_shobers}*{Definition 2.19}]
    \label{defn:smooth_proper}
    Let $\C$ be a dualizable $A$-linear category.
    \begin{enumerate}
        \item Say that $\C$ is \emph{smooth} if $\id_{\C}$ is left dualizable.
        \item Say that $\C$ is \emph{proper} if $\id_{\C}$ is right dualizable.
    \end{enumerate}
\end{defn}

\begin{rmk}~
\begin{enumerate}
    \item 
    The $\infty$-category $\rm{Mod}_A$ is rigid over the $\infty$-category of spectra $\rm{Sp}$.
    Rigidity implies that the right adjoint of a $\rm{Mod}_A$-linear functor is itself $\rm{Mod}_A$-linear \cite{GR1}*{Ch.~1, Lemma 9.3.6}, so we could have defined left/right dualizable endofunctors without the adjective $A$-linear.
    \item 
    The functor $\Xi_{\C}^{-1}(\id_{\C})$ is nothing but $\rm{coev}_{\C/A}$, and the functor $\Theta_{\C}^{-1}(\id_{\C})$ is nothing but $\rm{ev}_{\C/A}$.
\end{enumerate}
\end{rmk}

There is one more finiteness condition that we wish to consider.
\begin{defn}
    \label{defn:finite_type}
    Let $\C$ be a compactly generated $\V$-module.
    Say that $\C$ is \emph{finite type} if it is a compact object of $\rm{Pr}^L_{\omega,\V}$.
\end{defn}

\begin{rmk}
    Suppose that $\C \in \rm{Pr}^L_{\omega,A}$ is a compactly generated $A$-linear category and that $f: A \to B$ is a morphism of connective $\EE_{\infty}$-$\CC$-algebras.
    If $\C$ is finite type, i.e. compact in $\rm{Pr}^L_{\omega,A}$, then its pullback $\C\otimes_A B$ is finite type, since the pullback functor $\rm{Pr}^L_{\omega,A} \to \rm{Pr}^L_{\omega,B}$ preserves compact objects. 
\end{rmk}

There are various implications between the finiteness conditions that we have introduced so far.
In the remainder of this subsection, we only record the implications that are sufficient for our purposes.
The reader interested in the other implications may consult \cite{toen_vacquie_moduli_of_objects}.
\begin{prop}
\label{prop:dualizable_implies_compact}
    Let $\V \in \rm{CAlg}(\rm{Pr}^L_{\omega})$ be a compactly-generated presentable symmetric monoidal category.
    Then the following statements hold:
    \begin{enumerate}
        \item The category $\rm{Pr}^L_{\omega,\V}$ is compactly generated, i.e. $\aleph_0$-presentable, and symmetric monoidal.
        \item The symmetric monoidal structure on $\rm{Pr}^L_{\omega,\V}$ preserves small colimits separately in each variable.
        \item The symmetric monoidal unit of $\rm{Pr}^L_{\omega,\V}$ is compact.
        \item Dualizable objects of $\rm{Pr}^L_{\omega,\V}$ are compact.
    \end{enumerate}
    \begin{proof}
        For (1) and (2), note that $\rm{Pr}^L_{\omega}$ is $\aleph_0$-presentable \cite{scholze_gestalten}*{Theorem 2.9} and symmetric monoidal \cite{scholze_gestalten}*{Definition/Proposition 3.6}.
        Therefore by \cite{lurie_HA}*{\S 4.2, \S 4.5}, we know that $\rm{Pr}^L_{\omega,\V}$ is $\aleph_0$-presentable and symmetric monoidal.
        Furthermore, the symmetric monoidal structure on $\rm{Pr}^L_{\omega}$ preserves small colimits separately in each variable, so the same is true for the symmetric monoidal structure on $\rm{Pr}^{L}_{\omega,\V}$ \cite{lurie_HA}*{4.8.5.21}.

        Now we prove item (3).
        Observe that \cite{ramzi_rigidity}*{Example 4.38} implies that the free functor $\rm{Pr}^L_{\omega} \to \rm{Pr}^L_{\omega,\V}$ endows $\rm{Pr}^L_{\omega,\V}$ with the structure of a rigid $\rm{Pr}^L_{\omega}$-algebra.
        In particular, the unit $\mathbf{1}_{\rm{Pr}^L_{\omega,\V}} = \V$ is $\rm{Pr}^L_{\omega}$-atomic, so it suffices to show that the symmetric monoidal unit of $\rm{Pr}^L_{\omega}$ is compact.
        This is done in Lemma \ref{lem:ani_compact}.

        Finally for (4), suppose that $\C \in \rm{Pr}^L_{\omega,\V}$ is dualizable.
        It is a general fact that for any symmetric monoidal $\infty$-category with colimit-preserving tensor product and compact unit, the dualizable objects are compact.
        (Indeed, the functor of taking maps out of a dualizable object is the composite of tensoring with the dual and then taking maps out of the unit).
        So we conclude (4) by applying items (2) and (3).
    \end{proof}
\end{prop}

\begin{lem}
\label{lem:ani_compact}
    The symmetric monoidal unit $\rm{Ani}$ of $\rm{Pr}^L_{\omega}$ is compact.
    \begin{proof}
        Let $\D$ be a compactly-generated presentable category.
        Recall that morphisms in $\rm{Pr}^L_{\omega}$ are functors that preserve small colimits and compact objects. 

        We claim that evaluation at the point determines an equivalence of $\infty$-categories
        \[
            \rm{Map}_{\rm{Pr}^L_{\omega}}(\rm{Ani},\D) \simeq \D^c,
        \]
        where the category on the right-hand side is the $\infty$-category of compact objects of $\D$.
        To see this, recall that $\rm{Ani}$ is the animation of the category of finite sets, i.e. $\rm{Ani}$ is equivalent to the $\infty$-category of finite-product-preserving functors $\rm{Fin}^{\op} \to \rm{Ani}$.
        Therefore, the $\infty$-category of colimit-preserving functors from $\rm{Ani}$ to $\D$ is equivalent, via restriction along the Yoneda embedding, to the $\infty$-category of functors from $\rm{Fin}$ to $\D$ that preserve finite coproducts.
        Since every finite set is equivalent to a set obtained as the finite coproduct of copies of the point, a functor $F \in \Map_{\rm{Pr}^L}(\rm{Ani},\D)$ is determined by $d := F(\rm{pt})$.

        Clearly, if $F \in \Map_{\rm{Pr}^L_{\omega}}(\rm{Ani},\D)$, then $d := F(\rm{pt})$ must be compact.
        On the other hand, suppose that $F$ preserves colimits and that $d = F(\rm{pt})$ is compact.
        Then for every finite set $K \in \rm{Ani}$, $F(K)$ is compact, since (1) $F$ commutes with colimits, (2) $K$ is a finite colimit of compact objects in $\rm{Ani}$, and (3) the compact objects in $\D$ are closed under taking colimits.
        Recall that an anima $X$ is compact if and only if it is the retract of a finite set $K$ \cite{kerodon}*{\href{https://kerodon.net/tag/067X}{TAG 067X}}. 
        Thus if $X$ is compact, then $F(X)$ is compact, since $F(X)$ is the retract of $F(K)$, and compact objects are stable under retracts \cite{kerodon}*{\href{https://kerodon.net/tag/064Z}{TAG 064Z}}.

        Finally, we recall that there is an equivalence of categories
        \[
            \theta: \rm{Pr}^L_{\omega} \isomto \rm{Rex}^{\aleph_0}_{\rm{ic}},\quad \D \mapsto \D^c,
        \]
        given by sending a compactly-generated presentable category to its subcategory of compact objects \cite{lurie_HTT}*{Propositions 5.5.7.8 and 5.5.7.10}.
        Thus we have
        \[
            \Map_{\rm{Pr}^L_{\omega}}(\rm{Ani},\rm{colim}_i\D_i) \simeq
            \theta(\rm{colim}_i\D_i) \simeq \rm{colim}_i\theta(\D_i) \simeq \Map_{\rm{Pr}^L_{\omega}}(\rm{Ani},\D_i),
        \]
        and conclude that $\rm{Ani}$ is compact in $\rm{Pr}^L_{\omega}$.
    \end{proof}
\end{lem}

\begin{prop}
    \label{prop:compact_generator}
    Let $\V = \rm{Mod}_A$, and let $\C$ be an $A$-linear category.
    Then the following statements hold.
    \begin{enumerate}
        \item If $\C$ is smooth, then $\C$ is equivalent to $\rm{RMod}_B$ for some smooth $\EE_1$-algebra $B$ over $A$.
        In particular, $\C$ has a compact generator.
        \item Suppose that $\C$ is equivalent to $\rm{RMod}_B$ for some $\EE_1$-algebra $B$ over $A$, and that $\C$ is finite type.
        Then $B$ is a compact object of $\rm{Alg}_{\EE_1}(\rm{Mod}_A)$.
    \end{enumerate}
    \begin{proof}
        The first statement is \cite{lurie_SAG}*{Proposition 11.3.2.4}.

        The second statement is proved in  \cite{toen_vacquie_moduli_of_objects}*{Lemma 2.11} in the language of model categories and dg categories.
        We translate this proof for completeness.
        There is an adjunction, with fully faithful left adjoint, between algebras over $A$ and pointed compactly-generated $A$-linear categories \cite{lurie_HA}*{4.8.5.12 et seq}:
        \begin{equation*}
            \rm{Alg}_A  \rightleftharpoons (\rm{Pr}^L_{\omega,A})_{\ast},\quad
            B \mapsto (\rm{RMod}_B,B),\quad
            \hom_{\D}(d,d) \leftmapsto (\D,d).
        \end{equation*}
        For every $(\D,d) \in (\rm{Pr}^L_{\omega,A})_{\ast}$, we have a fiber sequence
        \[
            \xymatrixcolsep{0.5cm}\xymatrix{
                \Map_{\rm{Alg}_A}(B,\hom_{\D}(d,d)) \ar[r] & \Map_{\rm{Pr}^L_{\omega,A}}(\rm{RMod}_B,\D) \ar[r] & \Map_{\rm{Pr}^L_{\omega,A}}(\rm{Mod}_A,\D),
            }
        \]
        obtained via the adjunction.
        
        Finally, suppose that $D = \rm{colim}_iD_i$ is a filtered colimit of $A$-algebras.
        Then the following diagram commutes, and the columns are fiber sequences:
        \[
            \xymatrix{
                \rm{colim}_i\Map_{\rm{Alg}_A}(B,D_i) \ar[d] \ar[r] & \Map_{\rm{Alg}_A}(B,D) \ar[d] \\
                \rm{colim}_i\Map_{\rm{Pr}^L_{\omega,A}}(\rm{RMod}_B,\rm{RMod}_{D_i}) \ar[d] \ar[r]^-{\sim} & \Map_{\rm{Pr}^L_{\omega,A}}(\rm{RMod}_B,\rm{RMod}_D) \ar[d] \\
                \rm{colim}_i\Map_{\rm{Pr}^L_{\omega,A}}(\rm{Mod}_A,\rm{RMod}_{D_i}) \ar[r]^-{\sim} & \Map_{\rm{Pr}^L_{\omega,A}}(\rm{Mod}_A,\rm{RMod}_D),
            }
        \]
        where the equivalences hold because $\C$ and $\rm{Mod}_A$ are compact objects of $\rm{Pr}^L_{\omega,A}$.
        Thus we conclude that $B$ is a compact object of $\rm{Alg}_A$.
    \end{proof}
\end{prop}

\subsection{Hochschild homology and negative cyclic homology}\hfill \\
We recall Hochschild homology and negative cyclic homology following \cites{hoyois_hochschild, hss}.
Throughout this section, let $\V \in \rm{CAlg}(\rm{Pr}^L)$ and let $\C$ be a $\V$-module.

\subsubsection{Definitions}
Then we may define a cyclic object $\C^{\natural} \in \rm{PSh}(\Lambda,\V)$ by
\[
    \C^{\natural}_n := \rm{colim}_{a_0,\ldots,a_n\in \C} \bigotimes_{i = 0}^n \C(a_i,a_{i+1}).
\]
Let $\tilde{\Lambda}$ be a groupoid completion of $\Lambda$, and let $\TT := \tilde{\Lambda}([0],[0])$ be the automorphism group of $[0]$.
We have an equivalence $B\TT \simeq \tilde{\Lambda}$.

The pullback functor
\[
    \rm{PSh}(\tilde{\Lambda},\V) \to \rm{PSh}(\Lambda,\V)
\]
has a left adjoint, which we denote by $\abs{\mhyphen}$.
Furthermore, the morphisms
\[
    \xymatrix{
        \rm{pt} \ar[r]^-{i} & B\TT \ar[r]^-{p} & \rm{pt}
    }
\]
define functors of presheaves: the functor $p_{\ast}(\mhyphen)$ corresponds to taking $\TT$-fixed points, and the composite $i^{\ast}\abs{\mhyphen}$ corresponds to taking the geometric realization of the underlying simplicial object \cite{hoyois_hochschild}*{Proposition 1.1}.

\begin{defn}
    \label{defn:hochschild_negative_cyclic}
    Let $\C \in \rm{Mod}_{\V}(\rm{Pr}^L)$.
    \begin{enumerate}
        \item The \emph{negative cyclic homology} of $\C$ is
        \[
            \rm{HN}(\C) := p_{\ast}\abs{\C^{\natural}} \in \V.
        \]
        \item The \emph{Hocschild homology} of $\C$ is
        \[
            \rm{HH}(\C) := i^{\ast}\abs{\C^{\natural}}.
        \]
    \end{enumerate}
\end{defn}
Since $p\circ i = \id$, the counit of the adjunction $p^{\ast}\dashv p_{\ast}$ induces a morphism
\[
    \rm{HN}(\C) \to \rm{HH}(\C).
\]

\subsubsection{Base change}
Now we specialize to the case when $\V = \rm{Mod}_A$ for a connective $\EE_{\infty}$-$\CC$-algebra $A$.
For the subsequent discussion, we use some rudiments of the theory of (affine) Stefanich rings, for which the reader may consult \cite{scholze_gestalten}.

We inductively define
\[
    0\rm{Pr}_{A} := \rm{Mod}_A,\quad (n+1)\rm{Pr}_{A} := \rm{Mod}_{n\rm{Pr}_A}(\rm{Pr}^L_{\omega}).
\]
For each $n \geq 0$, we have a symmetric monoidal functor
\[
    \rm{CAlg}_{\CC}^{\rm{cn}} \to \rm{CAlg}(\rm{Pr}^L_{\omega}),
\]
which sends morphisms $f: A \to B$ to pullbacks $f_n^{\ast}: n\rm{Pr}_A \to n\rm{Pr}_B$.
The Adjoint Functor Theorem \cite{lurie_HTT}*{5.5.2.9} implies that the functor $f_n^{\ast}$ admits a right adjoint $f_{n,\ast}$, and in this setting, the right adjoint is colimit-preserving and satisfies base change.
Furthermore, for $n \geq 1$ the functor $f_n^{\ast}$ has a left adjoint $f_{n,\sharp}$ that coincides with $f_{n,\ast}$.

We use this machinery to prove the following
\begin{prop}
    \label{prop:base_change_hochschild}
    Suppose that $\C \in \rm{Pr}^L_{\omega,A}$ is a compactly generated $A$-linear category, and suppose that $f: A \to B$ is a morphism of connective $\EE_{\infty}$-$\CC$-algebras.
    Then we have an equivalence of $B$-modules
    \[
        \rm{HH}(\C\otimes B) \simeq \rm{HH}(\C) \otimes B.
    \]
    \begin{proof}
        First we claim that there is an equivalence
        \[
            f_0\circ \C^{\natural} \simeq (f_1^{\ast}\C)^{\natural}
        \]
        of presheaves on $\Lambda$ valued in $\rm{Mod}_A$.
        Indeed, for each $[n] \in \Lambda$, we have an equality of colimits 
        \[
            \rm{colim}\bigotimes (f_1^{\ast}\C)(f_0^{\ast}a_i,f_0^{\ast}a_{i+1}) \simeq \rm{colim} \bigotimes (f_1^{\ast}\C)(b_i,b_{i+1}).
        \]
        This is because the functor 
        \[
            F := f_1^{\ast}\circ(\mhyphen): \rm{Fun}(\sqcup_{i=0}^n \Delta^0,\C) \to \rm{Fun}(\sqcup_{i=0}^n \Delta^0,f_1^{\ast}\C)
        \]
        between the two indexing categories is a right adjoint, 
        and thus right cofinal.
        These equivalences also give equivalences of morphisms, which are defined using the universal property of the colimit.
        At this point, we may conclude that
        \[
            \rm{HH}(f_1^{\ast}\C) := i^{\ast}\abs{(f_1^{\ast}\C)^{\natural}} \simeq i^{\ast}\abs{f_0^{\ast}(\C^{\natural})}.
        \]
        To finish the proof, we note that $f_0^{\ast}$ commutes with $\abs{\mhyphen}$ and with $i^{\ast}$, so we conclude that
        \[
            i^{\ast}\abs{f_0^{\ast}(\C^{\natural})} \simeq f_0^{\ast}i^{\ast}\abs{\C^{\natural}} =: f_0^{\ast}\rm{HH}(\C).
        \]
    \end{proof}
\end{prop}

\begin{rmk}
\label{rmk:base_change_negative_cyclic}
    Suppose that $\C \in \rm{Pr}^L_{\omega,A}$ is a compactly generated $A$-linear category, and suppose that $f: A \to B$ is a morphism of connective $\EE_{\infty}$-$\CC$-algebras.
    Then we have an equivalence
    \[
        \rm{HN}(f_1^{\ast}\C) := p_{\ast}\abs{(f_1^{\ast}\C)^{\natural}} \simeq p_{\ast}\abs{f_0^{\ast}(\C^{\natural})} \simeq p_{\ast}f_0^{\ast}\abs{\C^{\natural}}.
    \]
    Furthermore, the following diagram commutes:
    \[
        \xymatrix{
            \rm{Fun}(\rm{pt},\rm{Mod}_A) \ar[r]^-{p^{\ast}} \ar[d]_-{f_0^{\ast}\circ(\mhyphen)} & \rm{Fun}(B\TT,\rm{Mod}_A) \ar[d]^-{f_0^{\ast}\circ(\mhyphen)} \\
            \rm{Fun}(\rm{pt},\rm{Mod}_B) \ar[r]_-{p^{\ast}} & \rm{Fun}(\rm{pt},\rm{Mod}_B),
        }
    \]
    and $p^{\ast}$ has a right adjoint $p_{\ast}$, so by abstract nonsense we obtain a composite morphism
    \[
        \rm{HN}(\C)\otimes_A B \to \rm{HN}(\C\otimes_A B).
    \]
\end{rmk}

\subsubsection{Trace and other descriptions}
Suppose that $\C$ is a dualizable $A$-linear category.
Then we have an equivalence of $A$-modules
\begin{equation}
\label{eq:hochschild_ITO_trace}
    \rm{HH}(\C) \simeq \rm{ev}_{\C/A} \circ \tau \circ \rm{coev}_{\C/A}(\mathbf{1}_{\rm{Mod}_A}),
\end{equation}
see \cite{hss}*{\S 4.5}.
Suppose further that $\C$ is smooth in $\rm{Pr}^L_{\omega,A}$.
Then we may describe its Hochschild homology in terms of its left Serre functor \cite{christ_shobers}*{Lemma 2.29}:
\begin{equation}
\label{eq:hochschild_ITO_hom}
    \rm{HH}(\C) \simeq \Hom^{\rm{Mod}_A}_{\rm{end}(\C)}(\id^!_{\C},\id_{\C}).
\end{equation}

\subsection{Left Calabi--Yau structures}
We recall the definition of a left Calabi--Yau structure.
\begin{defn}[=\cite{brav_dyckerhoff_1}*{Definition 3.5}, \cite{christ_shobers}*{Definition 3.1}]
Let $\C \in \rm{Pr}^L_{\omega,A}$ be a smooth $A$-linear category.
A \emph{left $d$-Calabi--Yau structure on $\C$} is a negative cyclic homology class
\[
    \eta: A[d] \to \rm{HN}(\C),
\]
such that the induced Hochschild homology class (cf.~\ref{eq:hochschild_ITO_hom}) defines an equivalence
\[
    \id^!_{\C} \isomto \id_{\C}[-d].
\]
\end{defn}

\begin{rmk}
    Suppose that $\C \in \rm{Pr}^L_{\omega,A}$ has a left $d$-Calabi--Yau structure, and suppose that $f: A \to B$ is a finitely presented morphism of connective $\EE_{\infty}$-$\CC$-algebras.
    Then via pullback $f_0^{\ast}: \rm{Mod}_A \to \rm{Mod}_B$, we obtain a left $d$-Calabi--Yau structure on $\C\otimes_A B$, (see Proposition \ref{prop:base_change_hochschild} and Remark \ref{rmk:base_change_negative_cyclic}).
\end{rmk}

%% file: results/moduli_of_objects.tex
\section{Moduli of objects}\label{sect:moduli_of_objects}
In this section we review the definition of the moduli stack of (pseudo-perfect) objects and the Representability Theorem of To\"en--Vacqui\'e \cite{toen_vacquie_moduli_of_objects}.

Fix a connective $\EE_{\infty}$-$\CC$-algebra $A$ and a compactly generated $A$-linear category $\C$.
Recall that for any $\D \in \rm{Pr}^L_{A}$, we may identify the tensor product with a category of functors,
\[
    \C\otimes_A \D \simeq \rm{Fun}^{\rm{ex}}_A(\C^{c,\op},\D),
\]
where the right-hand side consists of $A$-linear functors that are exact. 

\begin{defn}
\label{defn:proper_support_pseudo_perfect}
    Suppose that $T$ is a scheme over $S = \Spec(A)$, and suppose we are given an object $F \in \C\otimes_A \rm{QCoh}(T)$.
    Write $\Phi_F$ for the corresponding $A$-linear exact functor from $\C^{c,\op}$ to $\rm{QCoh}(T)$.
    \begin{enumerate}
        \item We say that $F$ is \emph{properly supported over $T$} if the functor $\Phi_F$ factors through $\rm{Coh}(T)$.
        \item We say that $F$ is \emph{pseudo-perfect} if the functor $\Phi_F$ factors through $\rm{Perf}(T)$.
    \end{enumerate}
    Write $\C_T^{\rm{ps}}$ (resp.~$\C_T^{\rm{pp}}$) for the full subcategory of $\C_T$ consisting of properly supported (resp.~pseudo-perfect) objects.
\end{defn}

If $\C$ is smooth, then pseudo-perfect objects satisfy a more familiar finiteness condition.
\begin{prop}[\cite{toen_vacquie_moduli_of_objects}*{Lemma 2.8}]
\label{prop:smooth_pp_implies_compact}
    Suppose that $T = \Spec(B)$ is an affine scheme over $S = \Spec(A)$.
    If $\C$ is smooth, then any pseudo-perfect object $F \in \C_B$ is compact.
\end{prop}

We quote the following lemma for later use in the proof of Theorem \ref{thm:CY_str_to_sss}.
\begin{lem}[=\cite{christ_shobers}*{Lemma 2.25}]
\label{lem:christ_lemma}
    Let $\C \in \rm{Pr}^L_{\omega,A}$ be finite type, and let $B$ be a finitely presented commutative $A$-algebra.
    Then there is an equivalence of functors
    \[
        \Hom^{\rm{Mod}_B}_{\C_B}(\mhyphen_2,\mhyphen_1)^{\vee} \simeq \Hom^{\rm{Mod}_B}_{\C_B}(\id^!_{\C_B}(\mhyphen_1),\mhyphen_2): (\C_B^{\rm{pp}})^{\op}\times\C^c \to \rm{Mod}_B.
    \]
\end{lem}

Now we may define the (laft) prestack of (pseudo-perfect) objects in $\C$.
\begin{defn}
    Let $A$ be a discrete, commutative $\CC$-algebra, and let $\C \in \rm{Pr}^L_{\omega,A}$ be a compactly generated $A$-linear category.
    Define the laft prestack of objects in $\C$ as the left Kan extension of the functor
    \[
        \tensor*[^{<\infty}]{\rm{SchAff}}{_{\rm{ft}/\Spec(A)}} \to \rm{Ani},\quad [T = \Spec(B)] \mapsto \C_B^{\rm{pp},\simeq}.
    \]
\end{defn}

\begin{thm}[=\cite{toen_vacquie_moduli_of_objects}*{Theorem 0.2}]
    Let $A$ be a discrete, commutative $\CC$-algebra, and let $\C \in \rm{Pr}^L_{\omega,A}$ be an $A$-linear category that is equivalent to $\rm{RMod}_B$ for some compact $A$-algebra $B \in \rm{Alg}_A$.
    Then the moduli stack of objects $\M_{\C/A}$ is a locally geometric derived stack locally of finite type.
    Moreover, the tangent complex at a point $x: \Spec(B) \to \M_{\C/A}$ corresponding to an object $F \in \C_B^{\rm{pp}}$ satisfies
    \[
        R\Gamma(\Spec(B),x^{\ast}\TT_{\M_{\C/A}}) \simeq \Hom^{\rm{Mod}_B}_{\C_B}(F,F)[1].
    \]
\end{thm}

%% file: results/forms.tex
\section{Forms and shifted symplectic structures}\label{sect:forms}
For the reader's convenience, we review \cite{brav_dyckerhoff_2}*{\S 5} on (closed) shifted $p$-forms and shifted symplectic structures. 
The only notable differences are that we use the language of stable linear $\infty$-categories and work over a discrete commutative $\CC$-algebra $A$ instead of a field of characteristic zero.
In particular, we review their construction of a shifted symplectic form on the moduli of objects given a left Calabi--Yau structure.

Throughout this section, we fix a discrete commutative $\CC$-algebra $A$ and we set $S = \Spec(A)$.

\subsection{Graded mixed complexes}

Recall from \cite{BZFN} that the affinization of the circle $S^1$ is $B\GG_{a,A}$, (we omit the subscript $A$ in the sequel), and that this induces an isomorphism
\[
    \rm{QCoh}(BS^1) \simeq \rm{QCoh}(B\GG_a).
\]
Furthermore, there is a right-split extension of group stacks
\[
    B\GG_a \to B\GG_a\rtimes\GG_m \to \GG_m,
\]
and passing to classifying stacks yields the following cartesian diagram:
\begin{equation}
    \xymatrix{
        B^2\GG_a \ar[r]^-{i} \ar[d]_-{\pi} & B(B\GG_a\rtimes\GG_m) \ar[d]_-{p} \\
        \Spec(A) \ar[r]_-{q} & B\GG_m \ar@/_1pc/[u]_-{j}.
    }
\end{equation}

Now we describe quasi-coherent sheaves on these stacks and the functors between them.
\begin{itemize}
    \item The category $\rm{QCoh}(B^2\GG_a)$ consists of $A$-modules with $S^1$-action.
    \item The category $\rm{QCoh}(B(B\GG_a\rtimes\GG_m))$ consists of mixed graded $A$-modules.
    \item The category $\rm{QCoh}(B\GG_m)$ consists of graded $A$-modules.
    \item The functor $\pi_{\ast}$ is taking $S^1$-invariants, i.e. taking negative cyclic homology, which we denote by $\rm{HN}(\mhyphen)$.
    \item The functor $p_{\ast}$ is taking weighted negative cyclic homology, which we denote by $\rm{HN}^{\rm{w}}(\mhyphen)$.
    \item The functor $j^{\ast}$ is forgetting the mixed structure.
    \item The functor $q^{\ast}$ (resp.~$i^{\ast}$) is taking the sum over the weight-graded components of a (mixed) graded $A$-module.
\end{itemize}
Recall that the morphism $q: \Spec(A) \to B\GG_m$ is perfect \cite{BZFN}*{Corollaries 3.22, 3.23}.
It follows that the morphism $i: B^2\GG_a \to B(B\GG_a\rtimes\GG_m)$ is perfect.
Therefore \cite{BZFN}*{Proposition 3.10} implies that the functors $q_{\ast},i_{\ast}$ are colimit-preserving, and thus admit right adjoints $q_{\ast}^R, i_{\ast}^R$ respectively.
These are given concretely by taking the product over the weight-graded components of (mixed) graded $A$-modules.
Moreover, there are natural transformations $q^{\ast} \to q_{\ast}^R$ and $i^{\ast}\to i_{\ast}^R$ given by mapping the direct sum to the product.
Finally, the fact that the following diagram commutes
\[
    \xymatrix{
        \rm{QCoh}(B^2\GG_a) \ar[r]^-{i_{\ast}} \ar[d]_-{\pi_{\ast}} & \rm{QCoh}(B(B\GG_a\rtimes\GG_m)) \ar[d]^-{p_{\ast}} \\
        \rm{QCoh}(\Spec(A)) \ar[r]_-{q_{\ast}} & \rm{QCoh}(B\GG_m),
    }
\]
implies that there is a natural transformation
\[
    \pi_{\ast}i_{\ast}^R \to q_{\ast}^Rp_{\ast}.
\]

As a consequence, we may observe the following 
\begin{prop}[=\cite{brav_dyckerhoff_2}*{Lemma 5.1} for $A$ a field of characteristic zero.]
    \label{prop:forms}
    Let $E \in \rm{QCoh}(B(B\GG_a\rtimes\GG_m))$ be a mixed graded $A$-module.
    Then there is a natural map
    \[
        \rm{HN}(i^{\ast}E) \to \prod_p \rm{HN}^{\rm{w}}(E)(p),
    \]
    which is given by the composition of functors
    \[
        \pi_{\ast}i^{\ast}\to \pi_{\ast}i_{\ast}^R \to q_{\ast}^Rp_{\ast}
    \]
    applied to $E$.
    Thus we obtain for each $p \in \ZZ$ a morphism of $A$-modules
    \[
        \rm{HN}(i^{\ast}E) \to \rm{HN}^{\rm{w}}(E)(p).
    \]
    Furthermore, we may apply the unit of the adjunction $j^{\ast}\dashv j_{\ast}$ to the above morphism, then take weighted pieces to obtain a morphism
    \[
        \rm{HN}^{\rm{w}}(E)(p) \to E(p).
    \]
\end{prop}

\subsection{Closed differential forms}
We define (closed) differential $p$-forms of degree $n$ on laft prestacks via Kan extension from affine schemes over $S$.

Suppose that $U \in \tensor*[^{<\infty}]{\rm{SchAff}}{_{\rm{ft}/S}}$ is an eventually coconnective affine scheme of finite type over $S = \Spec(A)$.
Then there is an equivalence
\[
    \maps(B\GG_a,U) \simeq \maps(S^1,U) =: \L U
\]
of stacks over $S$, \cite{BZFN}.
The scaling action of $\GG_m$ on $\GG_a$ induces a $(B\GG_a\rtimes\GG_m)$-action on $B\GG_a$, so we ultimately get a $(B\GG_a\rtimes\GG_m)$-action on global functions on the loop space of $U$, $\Gamma(\L U,\O)$.
This action yields a weight-space decomposition
\[
    \Gamma(\L U,\O) \simeq \bigoplus_p \Gamma(\L U,\O)(p).
\]
\begin{defn}
    Let $U \in \tensor*[^{<\infty}]{\rm{SchAff}}{_{\rm{ft}/S}}$ be an eventually coconnective affine scheme of finite type over $S = \Spec(A)$.
    \begin{enumerate}
        \item The \emph{space of $p$-forms of degree $n$ on $U$} is
        \[
            \A^p(U/S,n) := \abs{\Gamma(\L U,\O)(p)[n-p]}.
        \]
        \item The \emph{space of closed $p$-forms of degree $n$ on $U$} is
        \[
            \A^{p,\cl}(U/S,n) := \abs{\rm{HN}^{\rm{w}}(\Gamma(\L U,\O))(p)[n-p]}.
        \]
    \end{enumerate}
\end{defn}
Recall that the functor
\[
    \abs{\mhyphen}: \rm{Mod}_A \to \rm{Ani}
\]
is the composite of taking connective covers $\tau_{\geq 0}: \rm{Mod}_A \to \rm{Mod}_A^{\rm{cn}}$ followed by $\Omega^{\infty}: \rm{Mod}_A^{\rm{cn}} \to \rm{Sp}^{\rm{cn}} \to \rm{Ani}$.

From Proposition \ref{prop:forms}, we obtain a morphism
\[
    \A^{p,\cl}(U/S,n) \to \A^p(U/S,n)
\]
that is natural in $U$.
This defines a morphism of prestacks
\[
    \A^{p,\cl}(\mhyphen/S,n) \to \A^{p}(\mhyphen/S,n)
\]
on $\rm{PreStk}_{\rm{laft}/S}$ via Kan extension.

Now we recall the central construction of \cite{brav_dyckerhoff_2}*{\S 5.2}.
Suppose that $\X$ is a laft prestack over $S$.
We may tautologically write $\X$ as the colimit of affine $S$-schemes over $\X$
\[
    \X = \rm{colim}\,U.
\]
Then we have an equivalence
\begin{gather*}
    \A^{p,\cl}(\X/S,p-n) := \rm{maps}_S(\X,\A^{p,\cl}(\mhyphen/S,p-n)) \simeq \lim \A^{p,\cl}(U/S,p-n) \\
    \simeq \lim \abs{\rm{HN}^{\rm{w}}(\Gamma(\L U,\O))(p)[-n]}.
\end{gather*}

Recall that if $T$ is an $S$-scheme, then the $\infty$-category of pseudo-perfect objects  in $\C_T$ may be identified with $A$-linear exact functors $\C^{c,\op} \to \rm{Perf}(T)$.
Suppose that $\C^{c,\op}$ is equivalent to $\C^c$ from now on, (which holds for our main example of interest).
Then by taking ind-completions, we may identify pseudo-perfect objects of $\C_T$ with continuous adjunctions
\[
    \C \rightleftharpoons \rm{IndPerf}(T),
\]
i.e. an adjunction with colimit-preserving right adjoint.
In particular, there is a universal continuous adjunction
\[
    \C \rightleftharpoons \rm{IndPerf}(\M_{\C/A}).
\]

From the universal continuous adjunction, we obtain a commutative diagram
\[
    \xymatrix{
        \rm{HN}(\C) \ar[r] \ar[d] & \rm{HN}(\rm{IndPerf}(\M_{\C/A})) \ar[d] \\
        \rm{HH}(\C) \ar[r] & \rm{HH}(\rm{IndPerf}(\M_{\C/A})).
    }
\]
There is a natural map
\[
    \rm{IndPerf}(\M_{\C/A}) \to \lim\rm{IndPerf}(U) = \lim\rm{QCoh}(U),
\]
where the limit is taken over affine $S$-schemes $U$ over $\M_{\C/A}$,
and functoriality of negative cyclic and Hochschild homology yields morphisms
\begin{gather*}
    \rm{HN}(\rm{IndPerf}(\M_{\C/A})) \to \lim\rm{HN}(\rm{QCoh}(U)),\\
    \rm{HH}(\rm{IndPerf}(\M_{\C/A})) \to \lim\rm{HH}(\rm{QCoh}(U)).
\end{gather*}
The identifications
\[
    \rm{HH}(\rm{QCoh}(U)) \simeq \Gamma(\L U,\O),\quad \rm{HN}(\rm{QCoh}(U)) \simeq \Gamma(\L U,\O)^{S^1},
\]
yield, for each $p \in \ZZ$, a morphism
\[
    \rm{HN}(\C) \to \lim_{U \to \M_{\C/A}}\rm{HN}^{\rm{w}}(\Gamma(\L U,\O))(p).
\]

After shifting and passing to spaces, we obtain a commutative diagram
\begin{equation}
    \label{eq:bd_construction}
    \xymatrix{
        \abs{\rm{HN}(\C)[-n]} \ar[r]^-{\tilde{\kappa}_p} \ar[d] & \A^{p,\cl}(\M_{\C/A}/S,p-n) \ar[d] \\
        \abs{\rm{HH}(\C)[-n]} \ar[r]_-{\kappa_p} & \A^p(\M_{\C/A}/S,p-n).
    }
\end{equation}

\subsection{From left Calabi--Yau structures to shifted symplectic forms}
We state and prove our main theorem, which produces a $(2-d)$-shifted relative shifted symplectic structure on $\M_{\C/A}$ given a left $d$-Calabi--Yau structure on a finite type $A$-linear category $\C$.
This is a straightforward generalization of \cite{brav_dyckerhoff_2}*{Theorem 5.6(1)}.

\begin{thm}
\label{thm:CY_str_to_sss}
    Let $A$ be a discrete, commutative $\CC$-algebra.
    Let $\C$ be a compactly generated, smooth and finite type $A$-linear category, and let $\M_{\C/A}$ be the derived moduli stack of pseudo-perfect objects.
    Let $\eta: A[d] \to \rm{HN}(\C)$ be a left $d$-Calabi--Yau structure on $\C$.
    Then $\eta$ induces a $(2-d)$-shifted relative symplectic structure on $\M_{\C/A}$.
    \begin{proof}
        From the diagram (\ref{eq:bd_construction}) with $p = 2$, the morphism
        \[
            \eta: A[d] \to \rm{HN}(\C)
        \]
        induces a closed $2$-form of degree $2-d$ on $\M_{\C/A}$, which we denote by $\tilde{\kappa}_2(\eta)$.
        It only remains to show that the underlying $2$-form $\kappa_2(\eta)$, of degree $2-d$, is non-degenerate.
        In other words, we must check that $\kappa_2(\eta)$ induces an equivalence
        \[
            \TT_{\M_{\C/A}} \isomto \LL_{\M_{\C/A}}[2-d].
        \]

        Towards this end, suppose that $x: \Spec(B) \to \M_{\C/A}$ is a morphism that classifies a pseudo-perfect object $F \in \C_B^{\rm{pp}}$.
        Then $F$ is also compact in $\C_B$ (Proposition \ref{prop:smooth_pp_implies_compact}), and Lemma \ref{lem:christ_lemma} implies that we have an equivalence
        \[
            \Hom^{\rm{Mod}_B}(F,F[d])^{\vee} \simeq \Hom^{\rm{Mod}_B}(\id^!_{\C_B}(F[d]),F).
        \]
        
        The image of a class $\theta \in \rm{HN}_d(\C)$ in $\rm{HH}_d(\C)$ yields, for each $\Spec(B)$, a morphism of endofunctors
        \[
            \theta': \id^!_{\C_B}[d] \to \id_{\C_B},
        \]
        since $\C_B$ is smooth as a $\rm{Mod}_B$-module.
        The resulting $2$-form $\kappa_2(\theta')$ on $\M_{\C/A}$ induces a morphism
        \[
            x^{\ast}\TT_{\M_{\C/A}} \to x^{\ast}\LL_{\M_{\C/A}}[2-d]
        \]
        that is given by the composite
        \begin{equation}\label{eq:id_of_pairing}
            \xymatrix{
                \Hom^{\rm{Mod}_B}(F,F)[1] \ar[d]^-{(\mhyphen)\circ\theta'(F)} \\
                \Hom^{\rm{Mod}_B}(\id^!(F[d]),F)[1] \ar[r]^-{\sim} & \Hom^{\rm{Mod}_B}(F[d],F)^{\vee}[1] \ar[d]^-{\simeq} \\
                & (\Hom^{\rm{Mod}_B}(F,F)[1])^{\vee}[2-d].
            }
        \end{equation}
        Since $\eta$ is a left $d$-Calabi--Yau structure, the morphism $\eta'$ is an equivalence, which implies that $(\mhyphen)\circ\eta'(F)$ is an isomorphism for any pseudo-perfect object $F \in \C_B$.
        Thus, the pairing induced by $\kappa_2(\eta)$ is non-degenerate, and we have obtained a $(2-d)$-shifted symplectic structure on $\M_{\C/A}$.
    \end{proof}
\end{thm}

%% file: results/sheaves.tex
\section{Moduli of sheaves}\label{sect:sheaves}
Let $p: X \to S$ be a smooth, quasi-projective family of Calabi--Yau $d$-folds over an affine base.
Our main object of interest is the derived moduli stack $\M_{X/S}$ of flat families of properly supported coherent sheaves on $X/S$.
We explain how this is a substack of the moduli stack of pseudo-perfect objects in the $\rm{QCoh}(S)$-module category $\C = \rm{IndCoh}(X)$, and verify that $\C$ satisfies the hypotheses of Theorem \ref{thm:CY_str_to_sss}.
Finally, we discuss some specific examples of families $X/S$ with an eye towards future applications.

\subsection{Setup}\label{sect:geometric_setup}
We fix the following notation/setup for this section.
Fix a smooth, finite type $\CC$-algebra $A$, set $S := \Spec(A)$, and let $p: X \to S$ be a family of smooth quasi-projective Calabi--Yau $d$-folds.
In other words, the morphism $p$ is smooth and quasi-projective, and there is an isomorphism
\begin{equation}
\label{eq:volume_form}
    \O_X \isomto \Omega^d_{X/S}.
\end{equation}
Our $\infty$-category of interest is $\C := \rm{IndCoh}(X) \simeq \rm{QCoh}(X)$ considered as an $A$-linear category.

\subsection{Finiteness conditions}
We begin by describing a duality datum for $\C \in \rm{Pr}^L_{\omega,A}$ in terms of the six-functor formalism of ind-coherent sheaves developed in \cite{GR1}.

Both $X$ and $S$ are smooth over $\CC$, so the functors $\Psi_X$ and $\Psi_S$ are equivalences between their respective categories of ind-coherent and quasi-coherent sheaves.
Therefore \cite{BZFN}*{Theorem 4.7} implies that there is a canonical equivalence
\[
    \rm{IndCoh}(X)\otimes_{\rm{QCoh}(S)}\rm{IndCoh}(X) \isomto \rm{IndCoh}(X\times_S X).
\]

The morphism $p: X \to S$ is smooth, hence Gorenstein.
Thus, the relative dualizing sheaf $\omega_{X/S}$ identifies with
\[
    \omega_{X/S}:= p^{\rm{QCoh},!}\O_S\simeq \Omega^d_{X/S}[d].
\]
Furthermore, the fact that $p$ is Gorenstein implies that the natural transformation
\[
    \Xi_X\circ p^{\rm{QCoh},!} \to p^! \circ \Xi_S,
\]
is an equivalence.

Consider the following putative duality datum:
\begin{gather}\label{eq:duality_datum}
\begin{split}
    &\rm{coev}_{X/S}: \\
    &\xymatrix{
        \rm{QCoh}(S) \ar[r]^-{\Xi_S}_-{\sim} \ar[d]_-{p^{\rm{QCoh},!}} & \rm{IndCoh}(S) \ar[d]^-{p^!} \\
        \rm{QCoh}(X) \ar[r]^-{\sim}_-{\Xi_X} & \rm{IndCoh}(X) \ar[r]^-{(\Delta_{X/S})_{\ast}^{\rm{IndCoh}}} & \rm{IndCoh}(X\times_S X) \ar[d]^-{\simeq} \\
        & & \rm{IndCoh}(X)\otimes_{\rm{QCoh}(S)}\rm{IndCoh}(X).
    }\\
    &\rm{ev}_{X/S}: \\
    &\xymatrix{
        \rm{IndCoh}(X)\otimes_{\rm{QCoh}(S)}\rm{IndCoh}(X) \ar[d]_-{\simeq} \\
        \rm{IndCoh}(X\times_S X) \ar[r]^-{\Delta_{X/S}^!} & \rm{IndCoh}(X) \ar[r]^-{p_{\ast}^{\rm{IndCoh}}} & \rm{IndCoh}(S) \ar[d]^-{\Psi_S}_-{\simeq} \\
        & & \rm{QCoh}(S).
    }
\end{split}
\end{gather}
One may verify the triangle identities via the correspondences given by the following commutative diagram and its reflection along the line $y = -x$:
\begin{equation}
    \xymatrix{
        X \ar[r]^-{\Delta_{X/S}} \ar[d]_-{\Delta_{X/S}} & X\times_S X \ar[r]^-{\id\times p} \ar[d]^-{\id\times\Delta_{X/S}} & X \\
        X\times_S X \ar[r]^-{\Delta_{X/S}\times\id} \ar[d]_-{p\times\id}& X\times_S X\times_S X \\
        X.
    }
\end{equation}
Thus we see that $\C = \rm{IndCoh}(X)$ is dualizable in $\rm{Pr}^L_{\omega,A}$.
Proposition \ref{prop:dualizable_implies_compact} implies that $\C$ is a compact object of $\rm{Pr}^L_{\omega,A}$, and hence of finite type.

Notice that the coevaluation morphism (\ref{eq:duality_datum}) is given by the composite functor
\[
    \rm{coev}_{X/S} = (\Delta_{X/S})_{\ast}^{\rm{IndCoh}}\circ p^!\circ \Xi_S,
\]
and that 
\[
    p^! \simeq p^{\rm{IndCoh},\ast}[d].
\]
It follows that $\rm{coev}_{X/S}$ has a right adjoint given by the composite
\begin{equation}
    \label{eq:coev_right_adjoint}
    (\rm{coev}_{X/S})^R \simeq \Psi_S \circ p^{\rm{IndCoh}}_{\ast}[-d]\circ \Delta_{X/S}^!,
\end{equation}
and that $\C$ is a smooth $A$-linear category.

We record these observations as the following
\begin{prop}
    \label{prop:smooth_and_finite_type}
    Assume Setup \ref{sect:geometric_setup}.
    Then $\C$ is a smooth and finite type $A$-linear category.
    In particular, there exists a compact $A$-algebra $B \in \rm{Alg}_A$ such that $\C$ is equivalent, as an $A$-linear category, to the category $\rm{RMod}_B$.
    \begin{proof}
        We have already seen that $\C$ is smooth and finite type.
        Proposition \ref{prop:compact_generator} implies the second statement.
    \end{proof}
\end{prop}

\subsection{Hochschild homology}
The computation of Hochschild homology for $\C$ in $\rm{Mod}_A$ is entirely parallel to the computation in \cite{brav_dyckerhoff_1}*{\S 5.2}.
The only difference is that we work with the relative dualizing complex instead of the absolute dualizing complex.

We may use the description of Hochschild homology in terms of traces to compute that
\begin{multline}
    \rm{HH}(\C) \simeq p_{\ast}^{\rm{IndCoh}}\Delta_{X/S}^{!}(\Delta_{X/S})_{\ast}^{\rm{IndCoh}}\omega_{X/S} \\ 
    \simeq (X\times_S X \to S)_{\ast}^{\rm{IndCoh}}\hom_{X\times_S X}(\Delta_{\ast}^{\rm{IndCoh}}\O_X,\Delta_{\ast}^{\rm{IndCoh}}\omega_{X/S}).
\end{multline}
Observe that
\begin{equation*}
    \rm{HH}_i(\C) \simeq \rm{Ext}_A^{-i}(\Delta_{\ast}\O_X,\Delta_{\ast}\omega_{X/S}) \simeq \rm{Ext}_A^{d-i}(\Delta_{\ast}\O_X,\Delta_{\ast}\Omega^d_{X/S}).
\end{equation*}
Therefore, the ordinary $A$-module $\rm{HH}_i(\C)$ vanishes for all $i > d$, since $\O_X$ and $\Omega^d_{X/S}$ are ordinary sheaves.
Moreover, for $i = d$ we have
\[
    \rm{HH}_d(\C) \simeq \rm{Ext}_A^0(\Delta_{\ast}\O_X,\Delta_{\ast}\Omega^d_{X/S}) \simeq H^0(X,\Omega^d_{X/S}).
\]
Finally, \cite{brav_dyckerhoff_1}*{Lemma 5.10} implies that for all $i \geq d$, we have an isomorphism
\[
    \rm{HN}_i(\C) \isomto \rm{HH}_i(\C),
\]
(their proof is stated for the case when $A$ is a field of characteristic zero, but works over a discrete commutative ring containing $\QQ$).

\subsection{Left Calabi--Yau structure}
The left $d$-Calabi--Yau structure constructed in \cite{brav_dyckerhoff_1}*{\S 5.2.2} generalizes over a base.
\begin{prop}
\label{prop:cy_structure}
    Assume we are in Setup \ref{sect:geometric_setup}.
    Then any trivialization 
    \[
        \O_X \isomto \Omega^d_{X/S}
    \]
    gives a left $d$-Calabi--Yau structure on $\C = \rm{IndCoh}(X)$ as an $A$-linear category.
    \begin{proof}
        Suppose we have a trivialization 
        \[
            \rm{vol}_{X/S}: \O_X \isomto \Omega^d_{X/S}.
        \]
        We obtain a negative cyclic homology class $\eta$, of homological degree $d$ via the identifications
        \[
            \rm{HN}_d(\C) \simeq \rm{HH}_d(\C) \simeq \rm{Ext}_A^{0}((\Delta_{X/S})_{\ast}\O_X,(\Delta_{X/S})_{\ast}\omega_{X/S}[-d]) \simeq H^0(X,\Omega^d_{X/S}).
        \]
        Since $\Delta_{\ast}\rm{vol}_{X/S}$ is an equivalence, we see that $\eta$ is a left $d$-Calabi--Yau structure.
    \end{proof}
\end{prop}

\subsection{The moduli of objects}
We describe the moduli of pseudo-perfect objects for $\C = \rm{IndCoh}(X)$ as an $A$-linear category.
Note that Proposition \ref{prop:smooth_and_finite_type} implies that we may apply the To\"en--Vacqui\'e Representability Theorem to $\C$.

Let $T$ be an eventually coconnective finitely presented scheme over $S = \Spec(A)$.
Then \cite{BZNP}*{Theorem 3.0.2} implies that $\F \in \C_T^{\rm{ps}}$ if and only if $\F$ is coherent and has proper support relative to $T$.
(This fact is the reason for the terminology in Definition \ref{defn:proper_support_pseudo_perfect}).
Recall that this means there exists a closed immersion $Z \to X_T$ such that $\F|_{X_T\setminus Z}\simeq 0$ and the composite $Z \to T$ is proper.
Since $\rm{Perf}(X_T) \subset \rm{Coh}(X_T)$, each pseudo-perfect object is properly supported.
Moreover, Proposition \ref{prop:smooth_pp_implies_compact} implies that each pseudo-perfect object $\F \in \rm{QCoh}(X_T)$ is perfect, since $\rm{QCoh}(X_T)^c$ equals $\rm{Perf}(X_T)$ as subcategories of $\rm{QCoh}(X_T)$.

We prove the converse in the following
\begin{prop}
    \label{prop:perfect_and_proper_support_implies_pseudo-perfect}
    Assume that $p: X \to S$ is a smooth morphism of classical varieties over $\CC$.
    Let $T$ be an eventually coconnective scheme over $S$, and let $\F \in \rm{Perf}(X\times_S T)$ be a perfect complex that has proper support over $T$.
    Then $\F$ is pseudo-perfect.
    \begin{proof}
        We prove the statement for $T = S$, and the general case follows from this one by base change.

        Let $\F \in \rm{Perf}(X)$ be a perfect complex, and suppose that $Z$ is a derived scheme with a closed immersion $i: Z \to X$ such that $\F|_{X\setminus Z} \simeq 0$ and such that the composite $Z \to S$ is proper.
        We must show that for any $\G \in \rm{Perf}(X)$, the quasi-coherent sheaf $p_{\ast}\hom_X(\G,\F)$ lies in $\rm{Perf}(S)$.
        The sheaf $\hom_X(\G,\F)$ is perfect and moreover properly supported on $Z$ \cite{lurie_SAG}*{Corollary 7.1.5.6}.
        Therefore, it suffices to show that if $\E \in \rm{Perf}(X)$ is properly supported, then $p_{\ast}\E \in \rm{Perf}(S)$.

        The closed immersion $i: Z \to X$ induces a morphism $\hat{\imath}: X^{\wedge}_Z \to X$ from the formal completion of $X$ along $Z$ to $X$.
        The pullback functor $\hat{\imath}^{\ast}$ has a left adjoint $\hat{\imath}_?$, and the fact that $\E$ is properly supported on $Z$ implies that $\E \simeq \hat{\imath}_?\hat{\imath}^{\ast}\E$ \cite{halpernleistner_preygel}*{Theorem 2.2.3(i-ii)}.
        Furthermore, the object $\hat{\imath}^{\ast}\E$ is compact in $\rm{QCoh}(X^{\wedge}_Z)$.
        Thus, we are reduced to showing that if $\E \in \rm{QCoh}(X^{\wedge}_Z)$ is compact, then so is $p_{\ast}\hat{\imath}_?\E \in \rm{QCoh}(S)$.
        This is the content of the following Lemma \ref{lem:end_of_compactness_proposition}.
    \end{proof}
\end{prop}

\begin{lem}
    \label{lem:end_of_compactness_proposition}
    Suppose that $p: X \to S$ is a smooth morphism of classical varieties over $\CC$.
    Let $i: Z \to X$ be a closed immersion from a derived scheme of finite type such that the composite $Z \to S$ is proper.
    Write $X^{\wedge}_Z$ for the formal completion of $X$ along $Z$, and let $\E$ be a compact object in $\rm{QCoh}(X^{\wedge}_Z)$.
    Then the object $p_{\ast}\hat{\imath}_?\E$ is compact in $\rm{QCoh}(S)$.
    \begin{proof}
       Consider the space \begin{equation}\label{eq:colimit_inside_map}
            \Map_{\rm{QCoh}(S)}(p_{\ast}\hat{\imath}_?\E,\rm{colim}_i \F_i).
        \end{equation}
        The functor $\Xi_S: \rm{QCoh}(S) \to \rm{IndCoh}(S)$ is fully faithful \cite{gaitsgory_indcoh}*{Proposition 1.5.3}, so we have an equivalence
        \[
            \Map_{\rm{QCoh}(S)}(p_{\ast}\hat{\imath}_?\E,\rm{colim}_i \F_i) \simeq \Map_{\rm{IndCoh}(S)}(\Xi_Sp_{\ast}\hat{\imath}_?\E,\Xi_S\rm{colim}_i\F_i).
        \]
        We may ``move the functor $\Xi$ past the two pushforward morphisms'' via the following two equivalences
        \begin{equation*}
            \Xi_S \circ p_{\ast} \simeq p_{\ast}^{\rm{IndCoh}}\circ \Xi_X, \quad\quad
            \Xi_X \circ \hat{\imath}_? \simeq \hat{\imath}_{\ast}^{\rm{IndCoh}}\circ \Xi_{X^{\wedge}_Z},
        \end{equation*}
        (see \cite{gaitsgory_indcoh}*{Proposition 3.6.7} and \cite{GR_dgindschemes}*{\S 7.6.2} respectively).
        Thus the space (\ref{eq:colimit_inside_map}) is equivalent to the space
        \[
            \Map_{\rm{IndCoh}(S)}((p\circ\hat{\imath})_{\ast}^{\rm{IndCoh}}\Xi_{X^{\wedge}_Z}\E,\Xi_S\rm{colim}_i\F_i).
        \]
        The map $p\circ\hat{\imath}$ is ind-proper, so we have an adjunction $(p\circ\hat{\imath})_{\ast}^{\rm{IndCoh}} \dashv (p\circ\hat{\imath})^!$.
        Furthermore, we may ``move the functor $\Xi$ past the two pullback morphisms'' via the following two equivalences
        \begin{equation*}
            p^!\circ \Xi_S \simeq \Xi_X \circ p^{\rm{QCoh},!},\quad\quad 
            \hat{\imath}^!\circ \Xi_X \simeq \Xi_{X^{\wedge}_Z}\circ \hat{\imath}^{\ast},
        \end{equation*}
        (see \cite{gaitsgory_indcoh}*{Proposition 7.4.7} and \cite{GR_dgindschemes}*{Proposition 7.6.4} respectively).
        By the Adjoint Functor Theorem, the functors $\hat{\imath}^{\ast}$, $p^{\rm{QCoh},!}$ commute with colimits, since they have right adjoints $\hat{\imath}_{\ast}$ and $p_{\ast}\hom(\omega_{X/S},\mhyphen)$ respectively.

        There is an adjunction $\Upsilon_{X^{\wedge}_Z}\dashv \Xi^{\vee}_{X^{\wedge}_Z}$ with colimit-preserving right adjoint \cite{GR_dgindschemes}*{\S 7.6.2}, which dualizes to an adjunction $\Xi_{X^{\wedge}_Z}\dashv \Upsilon^{\vee}_{X^{\wedge}_Z}$ with colimit-preserving right adjoint.
        Therefore, the functor $\Xi_{X^{\wedge}_Z}$ preserves compact objects, so we may conclude that the space (\ref{eq:colimit_inside_map}) is equivalent to
        \begin{align*}
            \rm{colim}_i \Map_{\rm{IndCoh}(X^{\wedge}_Z)}(\Xi_{X^{\wedge}_Z}\E,\hat{\imath}^!p^!\Xi_S \F_i) &\simeq \Map_{\rm{IndCoh}(S)}(\Xi_S p_{\ast}\hat{\imath}_?\E,\Xi_S\F_i) \\
            &\simeq \Map_{\rm{QCoh}(S)}(p_{\ast}\hat{\imath}_?\E,\F_i).
        \end{align*}
        Thus we have shown that $p_{\ast}\hat{\imath}_?\E$ is compact.
    \end{proof}
\end{lem}

Proposition \ref{prop:perfect_and_proper_support_implies_pseudo-perfect} has the following
\begin{cor}
    \label{cor:pseudo_perfect_sheaves}
    Assume we are in Setup \ref{sect:geometric_setup}.
    Let $T$ be an eventually coconnective scheme over $S = \Spec(A)$.
    Then the pseudo-perfect objects of $\C_T = \rm{QCoh}(X\times_S T)$ are exactly the perfect complexes with proper support.
    \begin{proof}
        Proposition \ref{prop:perfect_and_proper_support_implies_pseudo-perfect} implies that perfect complexes with proper support are pseudo-perfect.
        On the other hand, pseudo-perfect objects are properly supported by \cite{BZNP}*{Theorem 3.0.2} and perfect since $\C$ is a smooth $A$-linear category.
    \end{proof}
\end{cor}

We may apply the Representability Theorem of \cite{toen_vacquie_moduli_of_objects} to conclude that the moduli of pseudo-perfect objects in $\C = \rm{IndCoh}(X)$ is geometric.
However, we want to add a flatness condition to obtain a geometric substack of $\M_{\C/A}$ whose classical truncation is the familiar moduli stack of flat families of coherent sheaves with proper support.
This has already been considered in e.g. \cite{porta_sala_2dCatHA}, and we recall the relevant details now.

\begin{defn}[=\cite{porta_sala_2dCatHA}*{Definition 2.1}]
    Let $p: X \to S$ be a morphism of derived Say that a connective almost perfect complex $\F \in \rm{APerf}(X)^{\rm{cn}}$ \emph{has tor-amplitude $\leq n$ with respect to $S$} if the following conditions holds:
    \[
        \pi_i(\F \otimes p^{\ast}\G) = 0
    \]
    for all $\G \in \rm{QCoh}(S)^{\heartsuit}$ and for all $i \notin [0,n]$.
    Say that $\F$ is \emph{flat over $S$} if $\F$ has tor-amplitude $\leq 0$ relative to $S$.
\end{defn}

\begin{notation}
    Assume we are in Setup \ref{sect:geometric_setup}.
    Let $\M_{X/S}$ be the derived moduli stack that sends an affine scheme $T$ over $S$ to the full subcategory of $\C_T^{\rm{pp},\simeq}$ consisting of complexes $\F$ that are connective and flat over $T$.
\end{notation}

\begin{prop}
    The following statements hold:
    \begin{enumerate}
        \item There is a natural morphism 
        \[
            \M_{X/S} \to \M_{\C/A}.
        \]
        \item The classical truncation of $\M_{X/S}$ is the classical stack of coherent sheaves on $X$ that are flat and properly supported over $S$.
        \item The prestack $\M_{X/S}$ is a geometric derived stack locally of finite type over $S$.
    \end{enumerate}
    \begin{proof}
        The first statement is obvious.

        For the second statement, apply the Derived Lazard Theorem \cite{lurie_HA}*{Theorem 7.2.2.15}.

        For the third statement, we appeal to Pridham's version \cite{pridham_representability}*{Theorem 2.17} of Lurie's Representability Theorem.
        Since the classical truncation of $\M_{X/S}$ is algebraic, it suffices to show that $\M_{X/S}$ is homotopically homogeneous.
        This is done in \cite{porta_sala_2dCatHA}*{\S 2.2, \S 2.3.1}.
        Note that in order to apply their results, we have used Corollary \ref{cor:pseudo_perfect_sheaves} to identify $\M_{\C/A}$ with the moduli stack of properly supported perfect complexes on $X$.
    \end{proof}
\end{prop}

\subsection{The relative shifted symplectic structure}
Proposition \ref{prop:cy_structure} implies that, whenever we are in Setup \ref{sect:geometric_setup}, we have a left $d$-Calabi--Yau structure on the $A$-linear category $\C = \rm{IndCoh}(X)$.
We may apply Theorem \ref{thm:CY_str_to_sss} to obtain a relative $(2-d)$-shifted symplectic structure on $\M_{\C/A}$.
There is a natural map of $S$-stacks $\M_{X/S} \to \M_{\C/A}$, and we may pull back the closed $2$-form on degree $2-d$ from the latter to the former.
In order to deduce that $\M_{X/S}$ has a $(2-d)$-shifted symplectic structure, we simply observe that the pulled back $2$-form remains non-degenerate.

Moreover, given any closed point $s \in S(\CC)$, the relative $(2-d)$-shifted symplectic form on $\M_{X/S}$ restricts to the usual $(2-d)$-shifted symplectic form on $\M_{X_s}$.
This is clear from the description of left $d$-Calabi--Yau structures on $\C$ and $\C\otimes_A \kappa(s)$ and from the base change properties of negative cyclic and Hochschild homology.

We may also consider substacks of $\M_{X/S}$ obtained by imposing semistability conditions and fixing topological invariants.
The shifted symplectic form on $\M_{X/S}$ pulls back to any substack obtained this way.

\subsection{Two kinds of examples}
Here we describe two kinds of examples of smooth, quasi-projective Calabi--Yau $3$-folds $X/S$ to which our main theorem applies. 
In both kinds of examples, we take $S$ to be the affine line.
We focus on families of CY3s because in a future work, we intend to use the relative symplectic structure on the moduli space of coherent sheaves to compare the enumerative geometries of the generic fiber $X_{\eta}$ and special fiber $X_0$.

\subsubsection{Classical deformation to the normal cone}
Fix a smooth, projective complex Calabi--Yau $3$-fold $Y$, and let $i: Z \to Y$ be the inclusion of a smooth, proper subvariety.
(Such a subvariety may be obtained as a complete intersection of hyperplanes on $Y$).
The Adjunction Formula implies that any such subvariety $Z$ on $Y$ satisfies
\[
    \det(N_{Z/Y}) \simeq \Omega^{\dim(Z)}_Z,
\]
so in particular the total space of the normal bundle $\NN_{Z/Y} := \Spec_Z(\rm{Sym}^{\bullet}N_{Z/Y}^{\vee})$ is a non-compact CY3.

Next we recall some details about the (classical) deformation to the normal cone.
More details can be found in \cite{fulton_intersection_theory}*{\S 5.1}.
The point of this discussion is to verify that the deformation to the normal cone, in this setting, is a family of smooth quasi-projective CY3s.

Define $\ol{X}$ to be the blowup of $Z\times\{0\}$ in $Y\times \AA^1$, and let $\pi: \ol{X} \to \AA^1$ be the composition of the blow-down map with the projection to $\AA^1$.
Note that $\ol{X}$ is smooth over $\CC$, since the center of the blow-up is smooth and $Y\times\AA^1$ is smooth.
Thus the map $\pi$ is flat and projective.

Let $\PP$ be the standard projective compactification of $\NN_{Z/Y}$, i.e. 
\[ \PP := \PP(N_{Z/Y}\oplus \O_Z) = \NN_{Z/Y} \sqcup \PP(N_{Z/Y}). \]
The exceptional divisor of the blowup $\ol{X}$ is $\PP$, since the normal cone to $Z\times\{0\}$ in $Y \times \AA^1$ is $N_{Z/Y}\oplus \O_Z$.
Recall that $\ol{X}_t \cong Y$ for $t \neq 0$ and
\[
    \ol{X}_0 \cong \PP \bigcup_{\PP(N_{Z/Y})} \rm{Bl}_Z(Y).
\]
Define $X := X\setminus(\rm{Bl}_Z(Y)\times\{0\})$, so that $\ol{X}_t = X_t$ for $t \neq 0$ and $X_0 = \NN_{C/Y}$.
Then $X \to \AA^1$ is still flat, but no longer projective as the central fiber is not proper. 

The deformation to the normal cone can also be described via the Rees algebra and relative spectrum construction.
Let $\I$ be the ideal sheaf of $Z$ in $Y$.
Define the quasi-coherent $\O_Y[t]$-algebra
\[
    S^{\bullet} := \bigoplus_{n\in \ZZ}t^n\I^{-n},\quad \I^{\leq 0} = \O_Y.
\]
Then $X$ is the relative spectrum $\Spec_{Y\times\AA^1}(S^{\bullet})$, and the composition 
\[
    \CC[t] \to \O_Y[t] \to S^{\bullet}
\]
corresponds to the blow-down map $X \to Y\times\AA^1$ followed by projection to $\AA^1$.
Note that $X\to\AA^1$ is a smooth morphism.

Finally, we describe the relative canonical bundle and volume form for $X/\AA^1$.
The formula for canonical bundles of blow-ups \cite{hart77}*{Ex.~II.8.5} implies that there is an isomorphism
\[
    \Omega^3_{\ol{X}/\CC} \cong (\ol{X}\to Y\times\AA^1)^{\ast}\Omega^4_{(Y\times\AA^1)/\CC}\otimes \O_{\ol{X}}((\rm{codim}_Y Z)\PP).
\]
The canonical bundle of $Y\times\AA^1$ is trivial, so we deduce that
\[
    \Omega^3_{\ol{X}/\CC} \cong \O_{\ol{X}}((\rm{codim}_Y Z)\PP).
\]
To compute the relative top-degree differentials $\Omega^3_{\ol{X}/\AA^1}$, recall that $\ol{X}$ and $\AA^1$ are both smooth over $\CC$ and that $\Omega^1_{\AA^1/\CC}$ is trivial. 
Hence by \cite{kleiman_relative_duality}*{Corollary 24} we see that
\[
    \Omega^3_{\ol{X}/\AA^1}\cong \O_{\ol{X}}((\rm{codim}_Y Z)\PP).
\]
We conclude that
\[
    \Omega^3_{X/\AA^1} \cong \O_{X}((\rm{codim_Y Z})\NN_{Z/Y}),
\]
by applying \cite{stacks-project}*{\href{https://stacks.math.columbia.edu/tag/0E30}{TAG 0E30}} and the fact that the relative dualizing complex of an open immersion is the structure sheaf of the open subscheme.
Note that $(\rm{codim}_Y Z)\NN_{Z/Y}$ is a principal Weil divisor on $X$ that corresponds to the principal Cartier divisor given by $t^{\rm{codim}_Y Z} \in \Gamma(X,\K^{\ast}_X)$.
Therefore, the line bundle $\Omega^3_{X/\AA^1}$ is trivializable, and we may choose a volume form to construct a relative $(-1)$-symplectic structure on $\M_{X/\AA^1}$.

We briefly discuss polarizations with which one may define semistability conditions.
If $L$ be an ample line bundle on $Y$, then $(X\to Y)^{\ast}L$ is relatively ample on $X/\AA^1$.
Indeed, this is because the composite $X \to Y\times\AA^1 \to Y$ is affine, hence quasi-affine, so we conclude using \cite{stacks-project}*{\href{https://stacks.math.columbia.edu/tag/0892}{TAG 0892},\href{https://stacks.math.columbia.edu/tag/01VK}{TAG 01VK}}.
So any line bundle $L$ on $Y$ gives us a relatively ample line bundle on $X$ with which we may define semistability conditions.

\subsubsection{1-parameter families of local curves}
Fix a smooth, complex projective curve $C$.
Let $L$ be a line bundle on $C$ (of high degree), and set $M := L^{-1}\otimes \Omega^1_C$ so that $\det(L \oplus M) \cong L\otimes M \cong \Omega^1_C$.
Choose a line $\AA^1 \to \rm{Ext}^1_C(L,M)$ through the origin.
Let $V$ be the pullback of the universal bundle on $C\times\rm{Ext}^1_C(L,M)$ to $C\times \AA^1$ and let $X$ be the total space of this bundle on $C\times\AA^1$.
Note that the determinant of $V$ is $\pr_C^{\ast}(\Omega^1_C)^{\vee}$.

Now we show that the line bundle of relative top-degree forms $\Omega^3_{X/\AA^1}$ is trivializable.
Observe that both of the natural maps below are smooth:
\[
    \xymatrix{
        X \ar[r]^-{p} & C\times \AA^1 \ar[r]^-{\pr_{\AA^1}} & C.
    }
\]
Therefore, we have a short exact sequence
\[
    0 \to p^{\ast}\Omega^1_{C\times\AA^1} \to \Omega^1_X \to \Omega^1_{X/(C\times\AA^1)} \to 0,
\]
and obtain a sequence of isomorphisms
\begin{equation}\label{eq:iso_1}
    \Omega^4_X \cong p^{\ast}\Omega^2_{C\times\AA^1} \otimes \Omega^2_{X/(C\times\AA^1)} \cong p^{\ast}\pr_C^{\ast}\Omega^1_C \otimes p^{\ast}\pr_C^{\ast}(\Omega^1_C)^{\vee} \cong \O_X,
\end{equation}
where we have used the fact that $\Omega^2_{C\times\AA^1} \cong \Omega^1_C\boxtimes \Omega^1_{\AA_1} = \pr_C^{\ast}\Omega^1_C$ and the fact that $\Omega^1_{X/(C\times\AA^1)}$ is isomorphic to the dual of $V$, since $V$ is isomorphic to the normal bundle of the zero-section.

We have another short exact sequence
\[
    0 \to p^{\ast}\pr_{\AA^1}^{\ast}\Omega^1_{\AA^1} \to \Omega^1_X \to \Omega^1_{X/\AA^1} \to 0,
\]
and thus we obtain an isomorphism
\begin{equation}\label{eq:iso_2}
    \Omega^4_X \cong \Omega^3_{X/\AA^1},
\end{equation}
since $\Omega^1_{\AA^1}$ is trivial.
Putting (\ref{eq:iso_1}) and (\ref{eq:iso_2}) together, we deduce that there exists a volume form
\[
    \Omega^3_{X/\AA^1} \simeq \O_X.
\]

If $L$ is an ample line bundle on $C$, then its pullback to $X$ via $X \to C\times\AA^1 \to C$ is a relatively ample line bundle with respect to $X/\AA^1$, (by the same argument used in the previous example). 
Therefore, we may use such a line bundle to define semistability conditions.